\documentclass[10pt]{amsart}%{article}
\usepackage{graphicx}
\usepackage{amsmath}
\usepackage{amsfonts}
\usepackage{amssymb}
\usepackage{amsthm}
\usepackage{verbatim}
\renewcommand{\-}{-\!\!\!\!\!\!\!\hspace{0.4mm}}

\newcommand{\lap}{\Delta}
\newcommand{\e}{\varepsilon}
\newcommand{\R}{\mathbb{R}}
\newcommand{\dd}{\partial}
\begin{document}
\renewcommand{\rmdefault}{cmr}
\newtheorem{lem}{Lemma}
\newtheorem{cor}{Corollary}
\newtheorem{thm}{Theorem}
\newtheorem{ass}{Assumption}
\newtheorem{prop}{Proposition}
\newtheorem{definition}{Definition}
\newtheorem{con}{Conjecture}
\newtheorem{rem}{Remark}
\title{Optimal Regularity for the No-Sign Obstacle Problem}

\author{John Andersson, Erik Lindgren and Henrik Shahgholian}

\thanks{H. Shahgholian has been partially supported  by the Swedish Research Council.
This research was carried out during the program Free boundaries at MSRI, Spring 2011, with support from NSF, Simons Foundation, and Viterbi Foundation.}

\maketitle

\begin{abstract}\noindent
In this paper we prove the optimal  $C^{1,1}(B_\frac12)$-regularity for a general obstacle type problem
$$
\lap u = f\chi_{\{u\neq 0\}}\textup{ in $B_1$},
$$
under the assumption that $f*N$ is $C^{1,1}(B_1)$, where $N$ is the Newtonian potential. This is the weakest assumption for which one can hope to get $C^{1,1}$-regularity. As a by-product of the $C^{1,1}$-regularity we are able to prove that, under a standard thickness assumption on the zero set close to a free boundary point $x^0$, the free boundary is locally a $C^1$-graph close to $x^0$, provided $f$ is Dini. This completely settles the question of the optimal regularity of this problem, that has been under much attention during the last two decades. 
\end{abstract}
\section{Introduction}
Our purpose in  this paper is two-fold. First to introduce a robust technique inspired by J. Andersson, H. Shahgholian and G. S. Weiss in \cite{ASW10}, to handle regularity questions for free boundary problems in general. Second, we want to apply the technique to resolve the regularity issue for the so-called no-sign obstacle problem, with the weakest possible assumptions on the right hand side.

We say that $u$ solves the no-sign obstacle type problem if for given $f$ such that $f*N\in C^{1,1}(B_1)$, where $N$ is the Newtonian potential, and for a reasonable smooth $g$
\begin{equation}\label{main}
\left\{\begin{array}{ll}
\Delta u= f\chi_{\{u\ne 0\}} & \textrm{ in } B_1, \\
u=g & \textrm{ on }\partial B_1,
\end{array}\right.
\end{equation}
in a suitable weak sense, where $\chi_A$ is the characteristic function of the set $A$ defined as
$$
\chi_A(x)=\left\{
\begin{array}{ll}
1 & \textrm{ if }x\in A, \\
0 & \textrm{ if }x\notin A,
\end{array}\right. 
$$
and where $B_r=\{x:|x|<r\}$. Since we are interested in the local regularity of $u$, the behavior of $g$ is not essential for our purposes.
It is noteworthy that equation (\ref{main}) is not different from the standard way of writing
$$
\Delta u =\chi_{\Omega}, \qquad u=\nabla u =0 \hbox{ in } \Omega^c .
$$
We have chosen this way of writing, just for simplicity, of exposition.

It is clear that a solution to (\ref{main}) is in general never better
than $C^{1,1}$, even if $f\in C^\infty(B_1)$.
This follows from the fact that $\Delta u$ is, in general, discontinuous 
across the free boundary $\Gamma=\partial (\operatorname{interior}\{u\ne 0\}\cap B_1$). 

This problem has been given a great deal of attention in the last two decades, after the seminal work of Sakai in \cite{sak91}, where he completely resolves the case $f=1$ in two dimensions. Until now the regularity of the solution when $f$ is allowed to be merely Dini continuous (see Definition \ref{dini}) has been an open problem. The main contribution of the paper is that we prove that the solution is indeed $C^{1,1}$ even under this weak assumption and that, under the standard thickness assumption on the zero set, the free boundary is locally a $C^1$-graph. %\mnote{dini eller c11?}

If $f\in L^p(B_1)$ for $p<\infty$ then we cannot hope to get better
regularity than $W^{2,p}$, which follows from standard Calderon-Zygmund theory
(see Theorem \ref{CZtheory} below).
Moreover, when $f\in L^\infty(B_1)$ then 
$f\chi_{\{u\ne 0\}}\in L^\infty(B_1)$ so that Calderon-Zygmund theory
implies that $u\in C^{1,\alpha}(B_{1/2})\cap W^{2,p}(B_{1/2})$ for
all $\alpha<1$ and $p<\infty$, but not for $\alpha=1$. Hence, it is clear that $f\in L^\infty(B_1)$ or even $f$ continuous is not strong enough to assure the $C^{1,1}$-regularity. The weakest possible assumption is clearly to ask the existence of $v\in C^{1,1}(B_1)$ so that $\lap v = f$. It is under this hypothesis (slightly weaker than the Dini condition), that we prove the $C^{1,1}$-regularity for the solution $u$ of \eqref{main}.

\subsection{Known result}
If one assumes that the solution $u$ is non-negative then it is classical that $u\in C^{1,1}(B_{1/2})$ if $f\in C^{\textrm{Dini}}$ ,
see for instance \cite{Blank} and \cite{CaffRev}. Moreover, in \cite{Blank} I. Blank proves that under a certain thickness assumption (see Theorem \ref{fbrthm}), the free boundary is a $C^1$-graph. That a function is $C^{\textrm{Dini}}$ is defined as follows.
\begin{definition} \label{dini}
We say that $f$ is Dini continuous, $f\in C^{Dini}(B_1)$, if there exist a 
modulus of continuity
(continuous and monotone non-negative function on $[0,1)$ that takes the value $0$ at the 
origin) $\sigma$ such that 
$$
|f(x)-f(y)|\le \sigma(|x-y|),
$$
for all $x,y\in B_1$ and
$$
\int_0^{1/2}\frac{\sigma(r)}{r}dr<\infty.
$$
\end{definition}
For the no-sign case and $f=1$,  it was proven by L.A. Caffarelli,
L. Karp and H. Shahgholian in \cite{CKS}, that solutions to (\ref{main}) are indeed
$C^{1,1}$ and that the free boundary is a $C^1$-graph close to free boundary points where the set $\{u=0\}$ is thick enough. This was later extended by 
H. Shahgholian
to cover the case when $f\in C^{0,1}$, see \cite{Sh}. The arguments in both
\cite{CKS} and \cite{Sh} are based on monotonicity formulas. These 
monotonicity formulas are probably optimal, in the sense that they cannot cover any weaker assumption on $f$, and it may therefore 
be difficult to improve on the regularity result in \cite{Sh} by using
the same methods. 

The best regularity result for the no sign obstacle problem is
due to A. Petrosyan and H. Shahgholian \cite{PS}, where they deduce
$C^{1,1}$ regularity of $u$ and the $C^1$ regularity of the free boundary, under a thickness assumption of the 
set $\{u=0\}$ (slightly stronger than the one in Theorem \ref{fbrthm}) when 
$f$ satisfies a double Dini condition:
that the modulus of continuity $\sigma$ of $f$ satisfies 
$$
\int_0^{1/2}\frac{\sigma(s)\ln(1/s)}{s}ds<\infty.
$$
In this article we will show the optimal regularity of solutions to the 
obstacle problem without any assumption on the sign nor do we need any 
a priori information on the set where $u=0$. Our main result is stated in the theorem below.
\begin{thm}[$C^{1,1}$-regularity]\label{mainthm}
Let $u$ be a solution to (\ref{main}) and assume furthermore that
$f= \Delta v$ where $v\in C^{1,1}(B_1)$ and that $g\in C(\partial B_1)$.
Then
$u\in C^{1,1}(B_{1/2})$ and
$$
\|D^2 u\|_{L^\infty(B_{1/2})}\le C\big( \|u\|_{L^1(B_1)}+\|D^2 v\|_{L^\infty(B_1)}\big),
$$
where $C$ depends on the dimension.
\end{thm} 
It is easy to see that if $f\in C^{\textrm{Dini}}(B_1)$ then our assumptions
are satisfied. Applying standard methods we obtain as a direct corollary the following regularity result on the free boundary which matches the 
result known for the case $f=1$ in \cite{CKS}.
\begin{thm}[Regularity of the free boundary]\label{fbrthm} Let $u$ be a solution to (\ref{main}) and assume in addition that
$f\in C^{\textup{Dini}}(B_1)$ and $f(0)=1$. Then there is a modulus of continuity $\sigma$ and $r_0>0$ such that if
$$
\frac{\operatorname{MD}(\{u=0\}\cap B_r)}{r}>\sigma(r),
$$ 
for some $r<r_0$ then for some $\rho>0$, $\partial( \operatorname{interior} \{u\neq 0\}\cap B_\rho)$ is a $C^1$-graph. Here $\operatorname{MD}$ stands for the minimal diameter.
\end{thm}

\begin{rem} In fact, all the arguments used to prove Theorem \ref{mainthm} work perfectly fine also for the case
$$
\lap u=f\chi_{\{|\nabla u|\neq 0\}}.
$$
The only place where one needs to be careful is in the proof of Proposition \ref{Mainprop}, where we claim that $D^2 u =0$ a.e. in the set $\{u=0\}$. 
But clearly, we also have $D^2u = 0$ a.e. in the set $\{|\nabla u|=0\}$. Moreover, Theorem \ref{fbrthm} also remains true in this case,
under the stronger assumption that
$$
\liminf_{\stackrel{x\to 0}{r\to 0}}\frac{\operatorname{MD}(\{u=0\}\cap B_r(x))}{r}>0.
$$
Since this quantity is stable with respect to perturbations, 
there is no longer any need of a Weiss type monotonicity formula. Therefore, one can, with arguments similar to those in the proof of 
Theorem \ref{fbrthm}, prove that the solution must be non-negative in a small neighborhood of the origin. Then the problem reduces to the obstacle
problem and the regularity result in \cite{Blank} applies.
\end{rem}

\begin{rem} There is nothing in these methods that restricts us to consider the Laplace equation and we believe that
we could derive Theorem \ref{mainthm} for more general linear operators. In particular we could consider
\begin{equation}\label{DiniCoeff}
Lu \equiv \partial_j \big( a^{ij}(x)\partial_i u \big)= f(x)\chi_{\{u\ne 0\}}
\end{equation}
for $a^{ij}\in C^{Dini}$ satisfying standard ellipticity conditions and $f$ as in Theorem \ref{mainthm}. In \cite{GruWi} estimates for the Green's potential for (\ref{DiniCoeff}) are derived and in \cite{Christ} BMO estimates for general kernels are proved. In \cite{Christ} the BMO estimates are proved only for H\"older kernels. Also, one would need to slightly refine the analysis in \cite{GruWi} in order to directly apply the Calderon-Zygmund theory as developed in \cite{Christ}. It would take us too far afield to reprove the results in \cite{GruWi} and \cite{Christ} in a form useful for our purposes. And for the sake of brevity and simplicity we have not attempted to pursue the greater generality. We believe, however, that such an extension of the theory is quite standard.
% It looks also plausible that the technique in this paper can be applied to second order linear elliptic operators, with Dini-coefficients. However, we  
%   did not make any efforts to investigate this, and have preferred to keep clarity of the exposition. We hope to come back to this sin the near future as well as
 %   the case of parabolic equations.
\end{rem}

\subsection{Notation and assumptions} Throughout the whole paper we will by $B_r(x_0)$ denote the ball of radius $r$ centered at the point $x_0$, 
and when there is no possible confusion we will write $B_r(0)=B_r$. Furthermore, we will use the notation $\Omega=\Omega(u)=\{u\ne 0\}\cap B_1$, 
for the free boundary $\Gamma=\Gamma(u)
= \partial (\Omega(u)^\circ) $. We also define the set \mbox{$\Lambda=\{u=0\}\cap B_1$}.
The  $BMO$ space is defined in Section \ref{BMOdef} and
the function $S(u,r,x^0)$ and the polynomial $p_{u,r,x^0}$ are defined in
Definition \ref{defofPi}. It is also convenient to introduce the notation
$$
\lambda_r = \frac{|\Lambda\cap B_r|}{|B_r|},
$$
and
$$
\delta(u,r)=\frac{\operatorname{MD}(\Lambda\cap B_r)}{r}, 
$$
since these quantities, as in many other free boundary problems, plays a central role. Finally, in the whole of this paper we will work under the following assumption:
\begin{ass}
In all of this paper, as in Theorem \ref{mainthm}, $f$ will be a function such that there is $v\in C^{1,1}(B_1)$ with $f=\lap v$, in the weak sense.
\end{ass}
As mentioned earlier, the assumption is fullfilled if for instance $f\in C^{\textup{Dini}}(B_1)$ (see Theorem 4.6 on page 60 and Problem 4.2 on page 70 in \cite{GT}).

We will also be somewhat sloppy with dimensional constants and we will in 
general denote all constants depending only on the dimension by $C$. At times 
we will indicate the dependence on another parameter and write $C_{p,n}$. We will also on occasion put a marker $C_0, C_1$ etc. on a constant in order to 
clarify some points, but the same constant may appear later in the paper 
without the subscript.

\section{Discussion and background}
Before we informally describe our main result let us recall the definition of $BMO$ spaces and state the Calderon-Zygmund estimates for
the Laplace equation.
\begin{definition}\label{BMOdef}
We say that a function $f\in L^2(\Omega)$ is in $BMO(\Omega)$ if
$$
\|f\|_{BMO(\Omega)}^2\equiv \sup_{x\in \Omega,r>0}
\frac{1}{r^n}\int_{B_r(x)\cap \Omega} |f(y)-(f)_{r,x}|^2+\|f\|_{L^2(\Omega)}^2<\infty,
$$
where $(f)_{r,x}$ is the average of $f$ in $B_r(x)\cap \Omega$.
\end{definition}
Next we state the Calderon-Zygmund estimates, for a proof see Theorem 9.11 on page 235 in \cite{GT}
and 6.3a on page 178 in \cite{SteinHar}.

\begin{thm}\label{CZtheory}
Let $\Delta w= f$ in $B_R$.
\begin{enumerate}
\item If $f\in L^p(B_R)$ for $1<p<\infty$ then
$$
\|D^2 w\|_{L^p(B_{R/2})}< C_{p,n}\big( \|f\|_{L^p(B_R)}+\|w\|_{L^1(B_R)}\big).
$$
\item\label{CZBMO} If $f\in L^\infty(B_R)$ then 
$$
\|D^2 w\|_{BMO(B_{R/2})}< C_{\infty,n}\big( \|f\|_{L^\infty(B_R)}+\|w\|_{L^1(B_R)}\big).
$$
\end{enumerate}
Here the constants depend only on $p$ and the space dimension.
\end{thm}
Our main idea is to exploit the limiting case in the Calderon-Zygmund theory.
Since $|\Delta u|\le C$ we know from \ref{CZBMO} in Theorem \ref{CZtheory}, that if $u\notin C^{1,1}$
then the singularity of $u$ comes from the average of the second derivatives.
If $(D^2 u)_{r,x^0}$ would be bounded then part 
\ref{CZBMO} in Theorem \ref{CZtheory} implies that the second derivatives of $u$ are bounded,
at least heuristically. 

%This part of  Theorem \ref{CZtheory} only used at two points; in the proof of Lemma \ref{anotherpointlesstechnicallity} and in the proof of Proposition \ref{Mainprop}. However, this usage seems to be crucial and we do not see any way to avoid it.

We also know that $u=0$, and thus $D^2 u=0$, on some open set close to the
free boundary. This implies that if $u$ is not $C^{1,1}$ then
$u(x)-x\cdot (D^2 u)_{r,0}\cdot x$ has to be very large
on the zero set of $u$ close to a singular point. But the Laplacian
of $u(x)-x\cdot (D^2 u)_{r,0}\cdot x-v$ will be bounded in
$L^2$ by Theorem \ref{CZtheory}, if $v\in C^{1,1}$. This $L^2$
bound will result in an $L^\infty$ bound on $(D^2 u)_{r,0}$. The details are spelled out in the next sections.

%It is interesting that the optimal regularity for no-sign obstacle problems
%is dependent on the limiting $BMO$ case $p=\infty$ of the Calderon-Zygmund theory.

\section{Preliminaries}
In this section we have gathered most of the technical lemmas needed 
in order to prove the main theorem. Most material in this section
is fairly trivial, although somewhat technical. The main idea and
the heart of the paper is Proposition \ref{Mainprop} in Section \ref{secproof} - the rest is no more
than trivial supporting technicalities. 
\begin{definition}\label{defofPi}
Let $\Pi(u,r,x^0)$ be the projection of $u$ into the space
of homogeneous harmonic polynomials in $B_r(x^0)$. In other words
$$
\int_{B_r(x^0)}\big| D^2u(x)-D^2\Pi(u,r,x^0)\Big|^2=\inf_{p\in \mathcal{P}_2}
\int_{B_r(x^0)}\big| D^2u(x)-D^2p(x)\big|^2,
$$
where $\mathcal{P}_2$ is the space of homogeneous harmonic polynomials of
second order and
$$
\big|D^2 p\big|^2=\sum_{i,j} \big|D^2_{i,j} p\big|^2.
$$
Moreover, we will use the notation
$$
\Pi(u,r,x^0)=S(u,r,x^0)p_{u,r,x^0}(x),
$$
where $p_{u,r,x^0}(x)$ is a second order homogeneous harmonic polynomial such that 

$$
\|D^2 p_{u,r,x^0}\|_{L^\infty(B_1)}=\sup_{B_1}|D^2 p_{u,r,x^0}|=1,
$$
and $S(u,r,x^0)\in \mathbb{R}_+$. 
\end{definition}
The following properties hold true for the projection.
\begin{lem}\label{trivialPiprop} Let $u$ be as in Theorem \ref{main}. Then
\begin{enumerate}
\item $\Pi(u,r,x^0)$ is well defined,
\item $\Pi(\cdot,r,x^0)$ is linear,
\item if $h$ is harmonic in $B_R(x^0)$ and $s,r<R$ then $\Pi(h,r,x^0)=\Pi(h,s,x^0)$,
%\item For $p,q\in [1,\infty]$ there exists constants $C_{p,q}$ depending,
%only on $p,q$ and the dimension such that NOT USED?!?!?!?
%$$
%\|\Pi(u,r,x^0)\|_{L^p(B_1)}\le C_{p,q}\|\Pi(u,r,x^0)\|_{L^q(B_1)},
%$$\mnote{REMOVE THE ONES WE DO NOT USE} 
%\item $\|\Pi(u,r,x^0)\|_{L^2(B_1)}\le \|u\|_{L^2(B_1)}$,??? NOT USED?!?!?!
\item $\|\Pi(u,r,x^0)\|_{L^2(B_1)}\le C\|D^2 u\|_{L^2(B_1)}$ for $r\in [\frac12,1]$,
\item $\|\Pi(u,r,x^0)\|_{L^\infty(B_1)}\le C\|D^2 u\|_{L^\infty(B_1)}$ for all $r\in (0,1)$.
\end{enumerate}
In the above, $C$ is a constant depending on the dimension.
\end{lem}
\begin{proof} Properties 1-3 are contained in Lemma 4.3 in \cite{ASW10}. The last two properties are consequences of the Poincar\'e inequality. \end{proof}

%All statements are trivial except the third property. which follows
%from the simple observation that we may expand $h$ as a series of 
%harmonic orthogonal polynomials $h(x)=\sum_{j=0}^\infty p_j(x-x^0)$. \end{proof}

\begin{lem}\label{BMOLem}
Let $u\in W^{1,2}(B_1)$ be a solution to (\ref{main}). Then for every $x^0\in \Gamma\cap B_{1/2}$ and $r<\frac14$, the following inequality holds
$$
\Big\|D^2\Big(\frac{u(rx+x^0)}{r^2}-\Pi(u,r,x^0)\Big) \Big\|_{L^2(B_1)}\le C\big( \|u\|_{L^1(B_1)}+\|\Delta u\|_{L^\infty(B_1)}\big),
$$
where $C$ depends on the dimension.
\end{lem}

\begin{proof} From the second part in Theorem \ref{CZtheory} it follows that
\begin{equation}\label{ivar}
\big\|D^2 u-(D^2 u)_{r,x^0}\|_{L^2(B_r(x^0))}\le 
C\big( \|u\|_{L^1(B_1)}+\|\Delta u\|_{L^\infty}(B_1)\big)r^{n/2}.
\end{equation}
We also observe that
$$
\-\int_{B_r(x^0)}D^2 u= \-\int_{B_r(x^0)}\left(\left(D^2 u-\frac{\lap u}{n}I\right) +\frac{\lap u}{n}I\right)=M(x^ 0,r)+\-\int_{B_r(x^0)}\frac{\lap u}{n}I, 
$$
where $M(x^0,r)$ is a constant matrix with zero trace and $I$ the identity matrix. Thus, if
$$
q_{x^0,r}=\frac12x^T M(x^0,r) x,
$$
then
$$
D^2q_{x^0,r}=M(x^0,r).
$$
It follows that
$$
(D^2 u)_{r,x^0}=\-\int_{B_r(x^0)}D^2 u=D^2 q_{x^0,r}+\-\int_{B_r(x^0)}\frac{\Delta u}{n}I.
$$
Hence, 
$$
\Big\|D^2u-D^2q_{x^0,r}\Big\|_{L^2(B_r(x_0))}\leq \Big\|D^2 u-\-\int_{B_r(x_0)}D^2 u \Big\|_{L^2(B_r(x_0))}+\Big\|\frac{\lap u}{n}I\Big\|_{L^2(B_r(x_0))}.
$$
From this, the definition of $\Pi$ and rescaling (\ref{ivar}), the 
lemma follows.\end{proof}

\begin{lem}\label{measuregaintoprojgain}
Let $u\in W^{1,2}(B_1)$ be a solution to (\ref{main}) and 
$f\in L^\infty(B_1)$. If 
$x^0\in \Gamma\cap B_{1/2}$ and $r<\frac14$ then there holds
\begin{align*}
\big\| \Pi(u,r,0))-\Pi(v,r,0)-\Pi(u,r/2,0)+\Pi(v,r/2,0)\big\|_{L^\infty(B_1)} \le C \|f\|_{L^\infty(B_1)}\lambda_r^{\frac{1}{2}},
\end{align*}
where $C$ depends only on $n$.
\end{lem}
\begin{proof} Without loss of generality we may assume $x^0=0$ and $r=1$. Moreover, we may assume that $v(0)=|\nabla v(0)|=\Pi(v,1,0)=0$,
if not then $u$ solves the same problem with 
$v(x)-\nabla v\cdot x-\Pi(v,1,0)$ in place of $v$. It is noteworthy that we only pay attention to what $\lap v$ is, so the linear part of $v$ does not matter. We write
$$
u(x)=v(x)+g(x)+\Pi(u,1,0)
$$
where $g$ solves
$$
\Delta g=-f(x)\chi_{\Lambda\cap B_1}.
$$
Then, by properties 2-3 in \mbox{Lemma \ref{trivialPiprop}}
\begin{equation}\label{Piudecomp}
\Pi(u,1/2,0)=\Pi(v,1/2,0)+\Pi(g,1/2,0)+\Pi(u,1,0).
\end{equation}
We need to estimate $\Pi(g,1/2,0)$. To this end we write $g=\tilde{g}+h$
where $\Delta h=0$ and $\tilde{g}$ is defined by the Newtonian potential
$$
\tilde{g}(x)=\frac{-1}{n(n-2)\omega_n}\int_{\mathbb{R}^n}\frac{f\chi_{\Lambda\cap B_1}(y)}{|x-y|^{n-2}}dy,
$$
where $\omega_n$ is the area of the $n$-dimensional sphere. Notice that 
$$\Pi(\tilde{g},1,0)=-\Pi(h,1,0),
$$
since $\Pi(v,1,0)=0$. Moreover, $\Pi(h,1/2,0)=\Pi(h,1,0)$ since $h$ is harmonic. In particular
\begin{equation}\label{ProjginHB}
\Pi(g,1/2,0)=\Pi(h,1/2,0)+\Pi(\tilde{g},1/2,0)=
-\Pi(\tilde{g},1,0)+\Pi(\tilde{g},1/2,0).
\end{equation}
But Calderon-Zygmund theory together with property 4 in Lemma \ref{trivialPiprop} imply that for $t\in [1/2,1]$
$$
\|\Pi(\tilde{g},t,0)\|_{L^2(B_1)}\le C_2\|f\chi_{\Lambda}\|_{L^2(B_1)},
$$
so for $t\in [1/2,1]$ we have
\begin{equation}\label{test}
\|\Pi(\tilde{g},t,0)\|_{L^\infty}\le C \|f\chi_{\Lambda}\|_{L^2(B_1)}\le
C \|f\|_{L^\infty(B_1)}\lambda_1^\frac12. 
\end{equation}
In particular, \eqref{test} holds for $t=\frac12$ and $t=1$. Using this, (\ref{Piudecomp}), (\ref{ProjginHB}) and that $\Pi(v,1,0)=0$, we conclude
\begin{align*}
\big\| \Pi(u,1,0)-\Pi(v,1,0)-\Pi(u,1/2,0)+\Pi(v,1/2,0)\big\|_{L^\infty(B_1)}\le C \|f\|_{L^\infty(B_1)}\lambda_1^{\frac{1}{2}}.
\end{align*}
This ends the proof of the lemma.\end{proof}

\section{$S$ bounded implies $C^{1,1}$}
In this section we describe through somewhat standard arguments that if $S(u,r,x)$ (the coefficients in front of the projection in Definition \ref{defofPi}) is uniformly bounded then we obtain $C^{1,1}$-regularity. The first lemma says that quadratic growth away from the zero set implies $C^{1,1}$-regularity.

\begin{lem}\label{QuadGrowimliesreg} 
Let $u\in W^{1,2}(B_1)$ be a solution to (\ref{main}).
Suppose that for all $y\in \Gamma\cap B_{1/2}$ and $r\in (0,1/4)$,
the following estimate holds
\begin{equation}\label{quadgrowth}
\sup_{B_r(y)}|u|\le C_0r^2.
\end{equation}
Then 
$$
\|D^2 u\|_{L^\infty(B_\frac12)}\le C \big(C_0+\|D^2 v\|_{L^\infty(B_1)}+\|u\|_{L^1(B_1)} \big),
$$
where $C$ depends only on $n$.
\end{lem}
\begin{proof} Let us first recall that for any harmonic function $w$ there holds (cf Theorem 7 on page 29 in \cite{Evans2009partial}):
\begin{equation}\label{harmest}\|D^2 w\|_{L^\infty(B_r)}\leq \frac{C}{r^{n+2}}\|w\|_{L^1(B_{2r})}\leq \frac{C}{r^{2}}\|w\|_{L^\infty(B_{2r})},
\end{equation}
where $C$ depends on the dimension.

Let now $x^0\in B_{1/2}$ and  
$r=\inf\big(1/4, \textrm{dist}(x^0,\Gamma)\big)$. If $r=1/4$ then
$u-v$ is harmonic in $B_{1/4}(x^0)$ and thus from \eqref{harmest} we can deduce
\begin{align}
&\|D^2(u(x)-v(x))\|_{L^\infty(B_{1/8}(x^0))}\nonumber\\
\label{Harest1} &\le C_n \big\|u(x)+v(x)-\nabla v(x^0)\cdot (x-x^0)-v(x^0)\big\|_{L^1(B_{1/4}(x^0))}\\
&\le C_n\big( \|u\|_{L^1(B_1)}+\big\|v(x)-\nabla v(x^0)\cdot (x-x^0)-v(x^0)\big\|_{L^1(B_1)}\big).\nonumber
\end{align}
Since 
$$
\|v(x)-\nabla v(x^0)\cdot (x-x^0)-v(x^0)\|_{L^1(B_1)}\le C\|D^2 v\|_{L^\infty(B_1)},
$$
it follows from (\ref{Harest1}) that
\begin{equation}\label{Harest3}
\|D^2 u\|_{L^\infty(B_{1/8}(x^0))}\le
C_n\big( \|u\|_{L^1(B_1)}+\|D^2 v\|_{L^\infty(B_1)}\big).
\end{equation}
If $r<1/4$ then (\ref{quadgrowth}) implies that
$$
\sup_{B_r(x^0)}|u|\le \sup_{B_{2r}(y)}|u|\le 4C_0 r^2,
$$
where we have chosen $y\in \Gamma$ such that 
$|x^0-y|=\textrm{dist}(x^0,\Gamma)$. Invoking \eqref{harmest} once more, we obtain that 
\begin{align}\nonumber
&\| D^2 (u-v)\|_{L^\infty(B_{r/2}(x^0))}\\&\le \frac{C}{r^2}\left(\sup_{B_r(x_0)}|u|+\sup_{B_r(x_0)}
\big| v(x)-\nabla v(x^0)\cdot (x-x^0)-v(x^0)\big|\right)\label{Harest2}\\
&\le C\big(4C_0+\|D^2v\|_{L^\infty(B_1)}\big).\nonumber
\end{align}
Inequalities (\ref{Harest1})-(\ref{Harest2}) together imply
that for every $x\in B_{1/2}$ we get a bound of the second derivatives
of $u$. This implies the lemma. \end{proof}

The following lemma proves that $S(u,r,x)$ being bounded implies quadratic growth for $u$ away from the free boundary.

\begin{lem}\label{anotherpointlesstechnicallity}
Assume that $u$ solves (\ref{main}) and that $x^0\in \Gamma\cap B_{1/2}$.
Assume furthermore that, for some $C_1$, $S(u,r,x^0)\le C_1$ for all $r\in (0,1/4)$. Then
$$
\sup_{B_r(x^0)}|u|\le C\big( C_1+ \|u\|_{L^1(B_1)}+\|\Delta u\|_{L^\infty(B_1)}\big)r^2,
$$
where $C$ depend only on the dimension.
\end{lem}
\begin{proof} By Lemma \ref{BMOLem} we know that
$$
\Big\|D^2 \Big(\frac{u(rx+x^0)}{r^2}-\Pi(u,r,x^0) \Big)\Big\|_{L^2(B_1)}
\le C\big( \|u\|_{L^1(B_1)}+\|\Delta u\|_{L^\infty(B_1)}\big).
$$
From the triangle inequality and the hypothesis of the lemma we can deduce
\begin{equation}\label{D2boundfru}
\Big\|D^2 \frac{u(rx+x^0)}{r^2}\Big\|_{L^2(B_1)}
\le C\big( C_1+ \|u\|_{L^1(B_1)}+\|\Delta u\|_{L^\infty(B_1)}\big).
\end{equation}
Define
$$
w(x)=\frac{u(rx+x^0)}{r^2}-x\cdot \left(\nabla  \left(\frac{u(rx+x^0)}{r^2} \right)\right)_{1,0}
-\left(\frac{u(rx+x^0)}{r^2} \right)_{1,0},
$$
where $(f(x))_{r,y}$ is the average of $f$ over the ball $B_r(y)$ as in Definition
\ref{BMOdef}. Then the Poincar\'e inequality and (\ref{D2boundfru}) imply that
\begin{equation}\label{ujnmki}
\|w\|_{L^2(B_1)}\le C\big( C_1+ \|u\|_{L^1(B_1)}+\|\Delta u\|_{L^\infty(B_1)}\big).
\end{equation}
Also $\Delta w=f(rx+x^0)\chi_{\{u(rx+x^0)=0\}}$ which together with (\ref{ujnmki})
imply that 
$$
\|w\|_{C^{1,\alpha}(B_{1/2})}\le C_\alpha \big( C_1+ \|u\|_{L^1(B_1)}+\|\Delta u\|_{L^\infty(B_1)}\big),
$$
for each $\alpha<1$. By assumption $u(x^0)=|\nabla u(x^0)|=0$, implying
$$
w(0)=-\left(\frac{u(rx+x^0)}{r^2}\right)_{1,0},\quad \textup{and}\quad \nabla w(0)=-\left(\nabla \left(\frac{u(rx+x^0)}{r^2} \right)\right)_{1,0}.
$$
By the $C^{1,\alpha}$ estimates for $w$ we can conclude
$$
\Big|\Big(\frac{u(rx+x^0)}{r^2}\Big)_{1,0} \Big|
\le C \big( C_1+ \|u\|_{L^1(B_1)}+\|\Delta u\|_{L^\infty(B_1)}\big),
$$
and
$$
\Big|\Big(\nabla \frac{u(rx+x^0)}{r^2} \Big)_{1,0} \Big|
\le C \big( C_1+ \|u\|_{L^1(B_1)}+\|\Delta u\|_{L^\infty(B_1)}\big),
$$
where $C$ is a constant depending only on the dimension.
From this and the triangle inequality we can conclude
\begin{align*}
\frac{1}{r^2}\sup_{B_{r/2}}|u|&\le \sup_{B_{1/2}}|w|+\Big|\Big(\frac{u(rx+x^0)}{r^2}\Big)_{1,0}\Big|
+\frac{1}{2}\Big|\Big(\nabla \frac{u(rx+x^0)}{r^2} \Big)_{1,0}\Big|\\
&\le C \big( C_1+ \|u\|_{L^1(B_1)}+\|\Delta u\|_{L^\infty(B_1)}\big),
\end{align*}
and the lemma is proved. \end{proof}

%%%%%%%%%%%%%%%%%%%%%%%%%%%%%%%%
%%%%  Proof of Theroem 1
%%%%%%%%%%%%%%%%%%%%%%%%%%%%%%%%

\section{Proof of Theorem 1}\label{secproof}

In this section we prove our main result. The first step is to prove that heuristically,
if $\Lambda$ does not have a cusp of infinite order, then $S(u,r,x)$ is uniformly bounded. In view of the previous section, 
this would imply the correct regularity in this special case.

\begin{prop}\label{Mainprop}
Let $u\in W^{1,2}(B_1)$ be a solution to (\ref{main}). Then there exist $C_0$ and $C_1$
depending only on the dimension 
such that if $x^0\in \Gamma\cap B_{1/2}$ and $r<\frac14$ then
$$
\frac{C_0\|D^2 v\|_{L^\infty(B_1)}}{S(r,u,x^0)-C_1\big( \|u\|_{L^1(B_1)}+\|D^2 v\|_{L^\infty(B_1)}\big)}\lambda_r^{1/2}\ge
\lambda_{\frac{r}{2}}^{1/2},
$$
whenever
$$
S(r,u,x^0)> 2C_1\big( \|u\|_{L^1(B_1)}+\|D^2 v\|_{L^\infty(B_1)}\big).
$$
\end{prop}
\begin{proof} We write
$$
\frac{u(rx+x^0)}{r^2}=w_r(x)+S(u,r,x^0)p_{u,r,x^0}(x)+g_r(x),
$$ 
where 
$$
\left\{
\begin{array}{ll}
\Delta g_r=-f(rx+x_0)\chi_{\Lambda(u(rx+x^0))} & \textrm{ in } B_{1}, \\
g_r=0 & \textrm{ on } \partial B_{1},
\end{array}\right.
$$
and 
$$
\left\{\begin{array}{ll}
\Delta w_r=f(rx+x_0) & \textrm{ in } B_{1}, \\
w_r=\frac{u(rx+x^0)}{r^2}-S(u,r,x^0)p_{u,r,x^0}(x) & 
\textrm{ on } \partial B_{1}. 
\end{array}\right.
$$
Now we claim the following: {\sl With $g_r$ and $w_r$ as above, the following estimates 
hold
\begin{equation}\label{L1weakforgr}
\| D^2 g_r\|_{L^2(B_{1/2})}\le 
C\|f\|_{L^\infty}\|\chi_{\Lambda(u(rx+x^0))}\|_{L^2(B_{1})},
\end{equation}
and
\begin{equation}\label{Linftyforwr}
\|D^2 w_r\|_{L^\infty(B_{1/2})}\le C\big( \|u\|_{L^1(B_1)}+\|D^2 v\|_{L^\infty(B_1)}\big).
\end{equation}}
It is clear that \eqref{L1weakforgr} follows from standard estimates for Laplace equation. Moreover, by Lemma \ref{BMOLem} 
\begin{equation}\label{d2est}
\Big\|D^2\Big(\frac{u(rx+x^0)}{r^2}-S(u,r,x^0)p_{u,r,x^0}(x)\Big) \Big\|_{L^2(B_1)}\le C\big( \|u\|_{L^1(B_1)}+\|D^2 v\|_{L^\infty(B_1)}\big). 
\end{equation}
Hence, if we define
\begin{align*}
\tilde{w}_r(x)&=\frac{u(rx+x^0)}{r^2}-S(u,r,x^0)p_{u,r,x^0}(x)\\&-\left(\frac{u(rx+x^0)}{r^2}\right)_{1,0}-x\cdot \left(\nabla \left(\frac{u(rx+x^0)}{r^2}\right)\right)_{1,0}, 
\end{align*}
then Poincar\'es inequality and \eqref{d2est} imply
$$
\|\tilde{w}_r\|_{L^2(B_1)}\leq C\big( \|u\|_{L^1(B_1)}+\|D^2 v\|_{L^\infty(B_1)}\big). 
$$
Then interior estimates for Laplace equation imply
$$
\|\tilde{w}_r\|_{C^{1,\alpha}(B_\frac34)}\leq C\big( \|u\|_{L^1(B_1)}+\|D^2 v\|_{L^\infty(B_1)}\big),
$$
and in particular
\begin{align*}
&\sup_{B_\frac34}|\tilde{w}_r|+|\tilde{w}_r(0)|+|\nabla \tilde{w}_r(0)|\\&=\sup_{B_\frac34}|\tilde{w}_r|+\Big|\left(\frac{u(rx+x^0)}{r^2}\right)_{1,0}\Big|+\Big| \left(\nabla \left(\frac{u(rx+x^0)}{r^2}\right)\right)_{1,0}\Big|\\&\leq
C\big( \|u\|_{L^1(B_1)}+\|D^2 v\|_{L^\infty(B_1)}\big).
\end{align*}
Therefore, 
\begin{align*}
\sup_{\dd B_\frac34} |w_r|&\leq \sup_{\dd B_\frac34}|\tilde{w}_r|+\Big|\left(\frac{u(rx+x^0)}{r^2}\right)_{1,0}\Big|+\Big| \left(\nabla \left(\frac{u(rx+x^0)}{r^2}\right)\right)_{1,0}\Big|\\&
\leq C\big( \|u\|_{L^1(B_1)}+\|D^2 v\|_{L^\infty(B_1)}\big).
\end{align*}
The estimate \eqref{Linftyforwr} now follows from interior estimates.

Let $\Lambda_r=\{u(rx+x^0)=0\}$. Since $u=0$ in $\Lambda$ we have that
$$
0=\Big\|D^2 \frac{u(r x+x_0)}{r^2}\Big\|_{L^2(\Lambda_r \cap B_{1/2})}=\big\|D^2\big(w_r-g_r+S(u,r,x^0)p_{u,r,x^0}\big)\big\|_{L^2(\Lambda_r\cap B_{1/2})}.
$$
From H\"olders inequality it follows that
\begin{align}
0&=\big\|D^2\big(w_r-g_r+S(u,r,x^0)p_{u,r,x^0}\big)\big\|_{L^2(\Lambda_r\cap B_{1/2})} \label{dmc}\\&\ge S(u,r,x^0)\|D^2 p_{u,r,x^0}\|_{L^2(\Lambda_r\cap B_{1/2})}
-\|D^2g_r\|_{L^2(\Lambda_r\cap B_{1/2})}-\|D^2  w_r\|_{L^2(\Lambda_r\cap B_{1/2})}.\nonumber
\end{align}
Next, since $D^2 p_{u,r,x^0}$ is a constant matrix,
$
\|D^2p_{u,r,x^0}\|_{L^\infty(B_1)}=1
$
implies
$$
\|D^2 p_{u,r,x^0}\|_{L^2(\Lambda_r\cap B_{1/2})}\geq C
|\Lambda_r\cap B_{1/2}|^{1/2},
$$
where $C$ is a dimensional constant. Using (\ref{Linftyforwr}),
%$$
%\|D^2 w_r\|_{L^\infty(B_{1/2})}\le C\big( \|u\|_{L^1(B_1)}+\|D^2 v\|_{L^\infty(B_1)}\big) 
%$$
it follows that
$$
\|D^2 w_r\|_{L^2(\Lambda_r\cap B_{1/2})}
\le C\big( \|u\|_{L^1(B_1)}+\|D^2 v\|_{L^\infty(B_1)}\big)|\Lambda_r\cap B_{1/2}|^{1/2}
$$ 
and also by \eqref{L1weakforgr}
\begin{align*}
\|D^2 g_r\|_{L^2(\Lambda_r\cap B_{1/2})}\le \|D^2 g_r\|_{L^2(B_{1/2})}
%&\le C\|f(r_j x+x^0)\chi_{\Lambda_r}\|_{L^2(B_1)}
\le C\|D^2 v\|_{L^{\infty}(B_1)} |\Lambda_r\cap B_{1}|^{1/2}.
\end{align*}
Observe that the right hand side in  \eqref{L1weakforgr}, can be estimated as follows
$$
\| f\|_{L^{\infty}(B_1)} = \| \Delta v \|_{L^{\infty}(B_1)} \leq \| D^2 v\|_{L^{\infty}(B_1)} .
$$

Inserting these three estimates in (\ref{dmc}) we deduce 
\begin{align*}
&C\big( \|u\|_{L^1(B_1)}+\|D^2 v\|_{L^\infty(B_1)}\big)|\Lambda_r\cap B_{1/2}|^\frac12+C\|D^2 v\|_{L^{\infty}(B_1)}|\Lambda_r \cap B_{1}|^\frac12 \\
&\ge S(u,r,x^0)|\Lambda_r\cap B_{1/2}|^\frac12.
\end{align*}
The lemma follows by simple algebra.\end{proof}

Now we are ready to prove our main theorem.
\begin{proof}[Proof of Theorem \ref{mainthm}] In view of Lemma \ref{QuadGrowimliesreg}-\ref{anotherpointlesstechnicallity} it is enough
to show that
\begin{equation}\label{tobenamed}
S(u,r,x^0) \le C\big( \|u\|_{L^1(B_1)}+\|D^2 v\|_{L^\infty(B_1)}\big),
\end{equation}
for all $r\in (0,1/8)$ and $x_0\in \Gamma\cap B_{1/2}$. We will do this by an iteration. Let us assume that
\begin{equation}\label{intcon}
S(u,r,x^0)=k_0\big( \|u\|_{L^1(B_1)}+\|D^2 v\|_{L^\infty(B_1)}\big),
\end{equation}
for some $k_0$ to be determined later and some fixed $r>0$. If (\ref{intcon})
is not satisfied for any $r>0$ then (\ref{tobenamed}) is certainly true. Furthermore, we will assume that 
$$
S(u,2^{-j}r,x^0)\ge k_0\big( \|u\|_{L^1(B_1)}+\|D^2 v\|_{L^\infty(B_1)}\big),
$$
for all $j=1,2,\ldots,j_0$ where $j_0$ is arbitrary and may be equal to $\infty$. We will show that then
\begin{equation}\label{conclusion}
S(u,2^{-j}r,x^0)\le Ck_0\big( \|u\|_{L^1(B_1)}+\|D^2 v\|_{L^\infty(B_1)}\big),
\end{equation}
for $j=1,2,\ldots,j_0$. Hence, for all $j=1,\ldots,\infty$, either 
$$
S(u,2^{-j}r,x^0)\le k_0\big( \|u\|_{L^1(B_1)}+\|D^2 v\|_{L^\infty(B_1)}\big),
$$
or \eqref{conclusion} holds. This clearly implies that $S(u,2^{-j}r,x^0)$ can never exceed 
$$Ck_0\big( \|u\|_{L^1(B_1)}+\|D^2 v\|_{L^\infty(B_1)}\big),
$$
which in turn implies (\ref{tobenamed}).

We notice that
\begin{align*}
I&:=\sup_{B_1}\big| \Pi(u,2^{-j}r,x^0)-\Pi(u,r,x^0)\big|\\
&\le \sup_{B_1}\Big| \sum_{k=1}^{j}\big(\Pi(u,2^{-k}r,x^0)-\Pi(u,2^{-k+1}r,x^0)\big)\\
&-\sum_{k=1}^{j}\big(\Pi(v,2^{-k}r,x^0)-\Pi(v,2^{-k+1}r,x^0)\big)\Big|\\
&+\sup_{B_1}\Big|\sum_{k=1}^{j}\big(\Pi(v,2^{-k}r,x^0)-\Pi(v,2^{-k+1}r,x^0)\big) \Big|\\
&\le \sum_{k=1}^{j}\sup_{B_1}\big|\Pi(u,2^{-k}r,x^0)-\Pi(v,2^{-k+1}r,x^0)\\
&-\Pi(u,2^{-k+1}r,x^0)+\Pi(v,2^{-k+1}r,x^0) \big|
+2C\|D^2 v\|_{L^\infty(B_1)},
\end{align*}
where we have used property 5 in Lemma \ref{trivialPiprop} in order to obtain the last inequality. Using Lemma \ref{measuregaintoprojgain} we arrive at
\begin{equation}\label{est}
I\leq C\|f\|_{L^\infty(B_1)}\sum_{k=1}^j\lambda_{2^{-k+1}r}^{1/2}+2C\|D^2 v\|_{L^\infty(B_1)}.
\end{equation}
Since 
$$
S(2^{-k+2}r)>k_0\big( \|u\|_{L^1(B_1)}+\|D^2 v\|_{L^\infty(B_1)}\big),
$$
for $k\geq 3$, it follows from Proposition \ref{Mainprop} that for $k\geq 3$ there holds
\begin{align*}
\lambda_{ 2^{-k+1}r}^\frac12\leq \frac{C\lambda_{2^{-k+2}r}^\frac12\|D^2 v\|_{L^\infty(B_1)}}{(k_0-C)\big( \|u\|_{L^1(B_1)}+\|D^2 v\|_{L^\infty(B_1)}\big)}.
%\left(\frac{|\Lambda\cap B_{2^{-k+2}r}(x^0)|}{|B_{2^{-k+2}r}|}\right)^{1/2}
%&\geq \left(\frac{|\Lambda\cap B_{2^{-k+1}r}(x^0)|}{|B_{2^{-k+1}r}|}\right)^{1/2}.
\end{align*}
This implies that we can estimate \eqref{est} and obtain
$$
I\le C_0\|D^2 v\|_{L^\infty(B_1)}\bigg(1+\sum_{k=1}^\infty 
\Big( C_1\frac{\|D^2 v\|_{L^\infty(B_1)}}{(k_0-C)\big(\|u\|_{L^1(B_1)}+\|D^2 v\|_{L^\infty(B_1)} \big)}\Big)^{k} \bigg).
$$
Notice that if $k_0$ is large enough, depending only on the dimension 
(say $k_0>2C_1+C$), then
the sum in the last expression converges and we may conclude that
$$
\sup_{B_1}\big| \Pi(u,2^{-j}r,x^0)-\Pi(u,r,x^0)\big|\le C\|D^2 v\|_{L^\infty(B_1)}.
$$
In particular, from the triangle inequality
\begin{align*}
S(u,2^{-j}r,x^0)&\le 
C\sup_{B_1}\big| \Pi(u,2^{-j}r,x^0)-\Pi(u,r,x^0)\big|+
S(u,r,x^0)\\
&\le k_0\big(\|u\|_{L^1(B_1)}+\|D^2 v\|_{L^\infty(B_1)}\big)+
C\|D^2 v\|_{L^\infty(B_1)},
\end{align*}
this clearly implies (\ref{conclusion}) and the theorem follows.\end{proof}

%%%%%%%%%%%%%%%%%%%%%%%%%%%%%%%%
%%%%  Proof of Theroem 2
%%%%%%%%%%%%%%%%%%%%%%%%%%%%%%%%

\section{Proof of Theorem 2}
In this section we use standard methods, adopted from for instance \cite{PS}, to prove that the free boundary is $C^1$ regular, except at cusp-like points. 
We will need some auxilary results presented in the appendix.

\begin{proof}[Proof of Theorem \ref{fbrthm}] %We will prove that under the hypotheses of the theorem, there is $s>0$ so that $u\geq 0$ in $B_s$. 
We prove that given $\e >0$ there is $r_\e$ such that if 
$ \delta(u,r)\geq \e $
for some $r<r_\e$, then $u\geq 0$ in $B_{r/2}$, and moreover $\delta(u,r)\geq 1/4$ for all $r\leq c_0 r$ for a universal $c_0$.
The latter follows by classification of global solutions.
Hence by Blank's regularity theory for the obstacle problem with
Dini continuous right hand side, Theorem 7.2 in \cite{Blank}, we conclude that the free boundary is $C^1$ in a yet smaller ball, with universal radius.
Finally we can take $\sigma$ as the inverse of the mapping $\e\to r_\e$. Let us now fill into details here below.

\vspace{3mm}

\noindent 
{\bf Step 1) $u\geq 0$ in $B_{r/2}$:} We argue by contradiction. If this is not the case, then there is a sequence 
$r_j\to 0$,  $x^j \in B_{r_j/2}$, $u_j, f_j$
(solving our problem)  such that, $f_j$ are uniformly Dini, $f_j(0)=1$, $\|u_j\|_{L^1}$ is uniformly bounded, and
\begin{equation}\label{deltaje}
\delta(u_j,r_j)\geq \e, \qquad u_j(x^j)<0 .
\end{equation}
Let now $W(u_j,r_j,0)$ be the monotonicity function introduced in the Appendix. Then in virtue of \eqref{deltaje}, we can apply  Proposition 1 in \cite{PS} 
(See Appendix), for $j$ large enough,  to conclude 
$$
W(u_j,r_j,0)<2A_n-\eta,
$$
for some $\eta=\eta(\e)$. This means in particular that $W(u_j,0^+,0)=A_n$ (see the paragraph before Theorem 2 in  \cite{PS}, and also Lemma 4).
Next by the upper semi-continuity of $W$, and a similar reasoning as above,  there is a small radius $\tau$ so that $W(u_j,0^+,y)=A_n$ for $y \in \Gamma (u_j)\cap B_\tau$. 
In particular, for $j$ large enough,  
$$
W(u_j,0^+,y^j) =  A_n ,
$$
where  $y^j \in \Gamma (u_j)$ is such that it realizes the distance 
$t_j:=\hbox{dist}(x^j,\Gamma (u_j))\to 0$ and $y^j\to 0$. This in turn implies
\begin{equation}\label{W-at-y}
W(u_j,t_j,y^j) <  (3/2)A_n  ,
\end{equation}
for $j$ large enough.

Letting
$$
v_j(x)=\frac{u_j(y^j+t_jx)}{t_j^2},
$$
we see that $v_j$ verifies
$$\lap v_j=f_j(\cdot\, t_j+y^j)\chi_{\{u_j\neq 0\}} \textup{ in $B_\frac{1}{2t_j}$},$$
$$ \sup_{B_p} |v_j|\leq C\rho^2,\textup{ for $1<\rho < 1/t_j$,}$$
\begin{equation*}%\label{nonneg}
v_j((x^j-y^j)/t_j)\leq 0,
\end{equation*}
\begin{equation*}%\label{lapvjf}
\lap v_j=f( \cdot\, t_j+y^j) \textup{ in }  B_1((x^j-y^j)/t_j) .
\end{equation*}
% \begin{equation}\label{min-diam-scaled} \delta(v_j, 1)\geq \epsilon. \end{equation}
All the above equations and inequalities, remain invariant in the limit and by standard estimates for elliptic equations that
 there is a subsequence, again labeled $v_j$, such that $v_j\to v_0$ where
$$
\lap v_0=f_0(0)\chi_{\{v_0\neq 0\}}=\chi_{\{v_0\neq 0\}}\textup{ in $\R^n$}, \qquad v_0(0)= 0,
$$
$$ \sup_{B_\rho} |v_0|\leq C\rho^2,\textup{ for $\rho>1$,}$$
\begin{equation*}%\label{nonneg1}
v_0(z^0)\leq 0,
\end{equation*}
\begin{equation*}%\label{pde-eq-1}
 \Delta v_0 = f_0(0)=1 \quad \hbox{ in } \quad B_1(z^0), \qquad z^0= \lim_j (x^j-y^j)/t_j, \quad |z^0|=1 .
\end{equation*}
Moreover, by (\ref{W-at-y}),
$$
W(v_0,1,0)=\lim_{j\to \infty}W(v_j,1,0)=\lim_{j\to \infty }W(u_j,t_j,y^j)  \leq (3/2)A_n,
$$
which implies that $v_0$ is a non-polynomial global solution (see the paragraph before Theorem 2 in  \cite{PS}, and also Lemma 4). %\mnote{varför inte använda hela klassificeringen?}

% Now, by using Proposition 1 in \cite{PS} again, we can conclude that 
%$$
%\liminf_{r\to 0}\delta(v_0,r)>0.
%$$
Hence we can apply  the classification theorem for global solutions (see Theorem II in \cite{CKS}) to conclude that $v_0 \geq 0$, $D_{ee}v_0 \geq 0$,
for all directions $e$. In particular the set $\{v_0=0\}$ is convex, and has non-empty interior ($v_0$ is non-polynomial) contradicting the facts that 
$$
v_0(0)=0,\qquad v_0(z^0) \leq 0 , \qquad \Delta v_0 = 1 \quad \hbox{in } B_1(z^0).
$$

%\begin{equation}\label{halfspace}
%v_0=\frac12 ((x_1-z_1)^+)^2,
%\end{equation}
%for some $z=(z_1,\ldots,z_n)$ and in some coordinate system. In addition, \eqref{lapvjf} implies
%$$
%\lap v = 1\textup{ in $B_1$},
%$$
%which means that $z_1<0$ so that in particular $v_0(0)>0$. But \eqref{nonneg} implies $v_0(0)\leq 0$, a contradiction.

\noindent{\bf Step 2) Applying Blank's regularity result:} 
From Step 1 we know that $u\geq 0$ in $B_{r/2}$. Now we claim that $\delta(u,r) \geq 1/4$ for $r<c_0r$ and $c_0$ a universal constant, whenever  $\delta(u,r) \geq \e$
for $r <r_\e$, and $r_\e$ small enough. If this fails, then once again as in Step 1, we shall have  sequence  $r_j \to 0$, $u_j$, ..., such that
$$
\delta (u_j,r_j) \geq \e , \qquad 
\delta (u_j,c_0 r_j) < 1/4 .
$$
Now scaling $u_j$ by $r_j$ and letting $r_j$ tend to zero we shall arrive at a non-negative global solution $v_0$ (as in the argument in Step 1)
where the minimal diameter for the the limit set $\{v_0=0\}$ satisfies
$$
\delta (v_0,1) \geq \e , \qquad 
\delta (v_0,c_0)< 1/4 ,
$$
for $c_0$ however small. On the other hand local regularity of L.A. Caffarelli \cite{CaffRev} implies that the free boundary is a $C^1$-graph locally,
in a uniform neighborhood of the origin. Hence there is a constant $c_0 >0$ such that 
the flatness condition $ \delta (v_0,c_0) > 1/4 $ should hold, contradicting the above conclusion.%\mnote{again why dont we use the explicit form of the global solution instead of local reg}

From here we can apply Blank's regularity result as in Theorem 7.2 in \cite{Blank}.

\end{proof}

\section{Appendix}
Here we present a version of Weiss' monotonicity formula and some consequences that we need in order to prove the regularity of the free boundary.
\begin{prop} \label{weissprop} Let $u$ be a solution of \eqref{main} and assume that $f\in C^{Dini}(B_1)$. Then there is a continuous function $F(r)$ with $F(0)=0$ such that
$$
F(r)+W(r,u,x^0),
$$
where
$$
W(r,u,x^0)=r^{-n-2}\int_{B_r(x^0)}\frac{|\nabla u|^2}{2}+fu \,dx+2r^{-n-3}\int_{\dd B_r(x^0)}u^2 d\sigma,
$$
is a monotonically increasing function for all $x^0\in \Gamma\cap B_\frac12$ and $r<\frac12$.
\end{prop}
\begin{proof} The proof is actually contained in the proof of Theorem M in \cite{PS}. The only difference is that there, the authors do not know that their solution is $C^{1,1}$. They are working under the assumption that $D^2 u$ is merely in $BMO$, so that $u$ grows in a $r^2\ln r$ fashion away from the free boundary. This is what forces the $\ln $-Dini assumption. If one removes the $\ln$ from their definition of $F$ in (12) on page 8, one can use this very $F$ in our case.\end{proof}

As mentioned in for instance \cite{PS}, the $C^{1,1}$-regularity (cf Theorem \ref{mainthm}) allows us to define the function
$$
W(u,0^+,x^0):=\lim_{r\to 0}W(u,r,x^0),
$$
for $x^0\in \Gamma$, which is upper semi-continuous in the $x^0$ variable. Moreover, as in Definition 3 in \cite{PS}, 
$W(u,0^+,x^0)$ can only attain two different values, 
$A_n$ and $2A_n$, where $A_n$ is a dimensional constant.

\bibliographystyle{plain}
\bibliography{ObstRegL2.bib}

\end{document}